\documentclass[12pt]{amsart}
\usepackage{amsfonts,amssymb,amscd,amsmath,enumerate,verbatim}
\usepackage[latin1]{inputenc}
\usepackage{amscd}
\usepackage{latexsym}
\usepackage{pstcol,pst-plot,pst-3d}
\usepackage{mathptmx}
\usepackage{multicol}

\psset{unit=0.7cm,linewidth=0.8pt,arrowsize=2.5pt 4}

\newpsstyle{fatline}{linewidth=1.5pt}
\newpsstyle{fyp}{fillstyle=solid,fillcolor=verylight}
\definecolor{verylight}{gray}{0.97}
\definecolor{light}{gray}{0.9}
\definecolor{medium}{gray}{0.85}

\def\frk{\mathfrak}               

\def\Phi{{\frk N}}
\def\opn#1#2{\def#1{\operatorname{#2}}} 
\opn\chara{char} \opn\length{\ell} \opn\pd{pd} \opn\rk{rk}
\opn\projdim{proj\,dim} \opn\injdim{inj\,dim} \opn\rank{rank}
\opn\depth{depth} \opn\grade{grade} \opn\height{height} \opn\bheight{bigheight}
\opn\embdim{emb\,dim} \opn\codim{codim}

\opn\Tr{Tr} \opn\bigrank{big\,rank}
\opn\superheight{superheight}\opn\lcm{lcm}
\opn\trdeg{tr\,deg}
\opn\reg{reg} \opn\lreg{lreg} \opn\ini{in} \opn\lpd{lpd}
\opn\size{size}\opn{\mult}{mult}\opn{\lex}{lex}
\opn\div{div} \opn\Div{Div} \opn\cl{cl} \opn\Cl{Cl}
\opn\Spec{Spec} \opn\Supp{Supp} \opn\supp{supp} \opn\Sing{Sing}
\opn\Ass{Ass} \opn\Min{Min}
\opn\Ann{Ann} \opn\Rad{Rad} \opn\Soc{Soc}
\opn\Syz{Syz} \opn\Im{Im} \opn\Ker{Ker} \opn\Coker{Coker}
\opn\Am{Am} \opn\Hom{Hom} \opn\Tor{Tor} \opn\Ext{Ext}
\opn\End{End} \opn\Aut{Aut} \opn\id{id} \opn\ini{in}

\opn\nat{nat}
\opn\pff{pf}
\opn\Pf{Pf} \opn\GL{GL} \opn\SL{SL} \opn\mod{mod} \opn\ord{ord}
\opn\Gin{Gin}
\opn\Hilb{Hilb}\opn\adeg{adeg}\opn\std{std}\opn\ip{infpt}
\opn\Pol{Pol}
\opn\sat{sat}
\opn\Var{Var}
\opn\Gen{Gen}
\opn\indmatch{indmatch}

\opn\aff{aff} \opn\con{conv} \opn\relint{relint} \opn\st{st}
\opn\lk{lk} \opn\cn{cn} \opn\core{core} \opn\vol{vol}
\opn\link{link} \opn\star{star}
\opn\gr{gr}

\def\Sc{{\mathcal S}}

\def\pot#1#2{#1[\kern-0.28ex[#2]\kern-0.28ex]}

\opn\dirlim{\underrightarrow{\lim}}
\opn\inivlim{\underleftarrow{\lim}}
\let\union=\cup
\let\sect=\cap

\let\Union=\bigcup

\let\to=\rightarrow

\def\Implies{\ifmmode\Longrightarrow \else
        \unskip${}\Longrightarrow{}$\ignorespaces\fi}
\def\implies{\ifmmode\Rightarrow \else
        \unskip${}\Rightarrow{}$\ignorespaces\fi}
\def\iff{\ifmmode\Longleftrightarrow \else
        \unskip${}\Longleftrightarrow{}$\ignorespaces\fi}

\let\:=\colon
\newtheorem{Theorem}{Theorem}[section]
\newtheorem{Lemma}[Theorem]{Lemma}
\newtheorem{Corollary}[Theorem]{Corollary}
\newtheorem{Proposition}[Theorem]{Proposition}

\newtheorem{Conjecture}[Theorem]{Conjecture}

\let\epsilon\varepsilon
\let\phi=\varphi
\let\kappa=\varkappa
\def\qed{\ifhmode\textqed\fi
      \ifmmode\ifinner\quad\qedsymbol\else\dispqed\fi\fi}
\def\textqed{\unskip\nobreak\penalty50
       \hskip2em\hbox{}\nobreak\hfil\qedsymbol
       \parfillskip=0pt \finalhyphendemerits=0}
\def\dispqed{\rlap{\qquad\qedsymbol}}

\opn\dis{dis}
\def\pnt{{\raise0.5mm\hbox{\large\bf.}}}

\opn\Lex{Lex}

\newcommand{\inD}[1][\relax]{\def\argone{#1}\def\temprelax{\relax}
  \ifx\argone\temprelax\right.\else\,\middle|#1\right.{}\fi}

\newif\ifbinary
\binarytrue

\begin{document}

\title{On the regularity of binomial edge ideals}

\author{Viviana Ene, Andrei Zarojanu}

\address{Viviana Ene, Faculty of Mathematics and Computer Science, Ovidius University, Bd.\ Mamaia 124,
 900527 Constanta, Romania
 \newline
 Simion Stoilow Institute of Mathematics of Romanian Academy, Research group of the project  ID-PCE-2011-1023,
 P.O.Box 1-764, Bucharest 014700, Romania} \email{vivian@univ-ovidius.ro}

\address{Andrei Zarojanu, Simion Stoilow Institute of Mathematics of Romanian Academy, Research group of the project  ID-PCE-2011-1023,
 P.O.Box 1-764, Bucharest 014700, Romania} \email{andrei\_zarojanu@yahoo.com} 
\thanks{The  authors were supported by the grant UEFISCDI,  PN-II-ID-PCE- 2011-3-1023.}

\begin{abstract}
We study the regularity of binomial edge ideals. For a closed graph $G$ we show that the regularity of the binomial edge ideal $J_G$ coincides with the regularity of 
$\ini_{\lex}(J_G)$  and  can be expressed in terms of the combinatorial data of $G.$ In addition, we give positive answers to Matsuda-Murai conjecture \cite{MM}
for some classes of graphs.
\end{abstract}
\maketitle

\section*{Introduction}

Binomial edge ideals were introduced in \cite{HHHKR} and \cite{Oh}. They are a generalization of the classical determinantal ideal generated by the $2$-minors
of a $2\times n$--matrix  of indeterminates. 

Let $S=K[x_1,\ldots,x_n,y_1,\ldots,y_n]$ be the polynomial ring in $2n$ variables over a field $K$. For $1\leq i < j\leq n$ we set $f_{ij}=x_iy_j-x_jy_i.$ Thus, $f_{ij}$ is the 
$2$-minor determined by the columns $i$ and $j$ in the $2\times n$--matrix $X$ with  rows $x_1,\ldots,x_n$ and $y_1,\ldots, y_n$.  Let $G$ be a simple graph on 
$[n]$ with edge set $E(G)$. The {\em binomial edge ideal} $J_G$ associated with $G$ is the ideal of $S$ generated by all the binomials $f_{ij}$ with $\{i,j\}\in 
E(G).$ For example, if $G$ is the complete graph on the vertex set $[n],$ then $J_G$ is equal to the ideal $I_2(X)$ of the maximal minors of $X$. 

In the last years, many properties  of binomial edge ideals have been studied in relation to the combinatorial data of the underlying graph and  applications to 
statistics have been investigated; see \cite{AR}, \cite{CR}, \cite[Chapter 6]{EH}, \cite{EHH}, \cite{HHHKR}, 
 \cite{RR},  \cite{Sara}, \cite{SZ},  \cite{Z}.

The Gr\"obner basis of a binomial edge ideal with respect to the lexicographic order  induced by $x_1>\cdots >x_n>y_1>\cdots> y_n$ was computed in \cite{HHHKR}. It 
turned out that this Gr\"obner basis is quadratic if and only if the graph $G$ is closed with respect to the given labeling which means that it satisfies the following 
condition: whenever $\{i,j\}$ and $\{i,k\}$ are edges of $G$ and either $i<j,\ i<k$ or $i>j,\ i>k,$ then $\{j,k\}$ is also an edge of $G.$ One calls a graph $G$ 
closed if it is closed with respect to some labeling of its vertices. Any closed graph is chordal and claw-free; see \cite{HHHKR}. In \cite{CR}, the authors show that actually $J_G$ has a quadratic Gr\"obner basis for some 
monomial order if and only if $G$ is closed. Combinatorial characterizations of closed graphs are given in \cite{CE} and \cite{EHH}. 

The regularity of binomial edge ideals has been studied in \cite{MM} and \cite{Sara}. In \cite[Theorem 1.1]{MM} it was proved that if $G$ is a connected graph, 
then $\ell\leq\reg(S/J_G)\leq n-1$ where $\ell$ denotes the length of the longest induced path in $G.$ We show in Theorem~\ref{reg} that if, in addition, $G$ is 
closed, then $\reg(S/J_G)=\reg(S/\ini_{\lex}(J_G))=\ell.$ In particular, it follows that the regularity of $J_G$ and $\ini_{\lex}(J_G)$ do not depend on the characteristic of the base 
field.

In \cite{MM}, the authors conjectured that if $G$ is connected, then  $\reg{S/J_G}=n-1$ if and only if $G$ is a line graph. It  straightforwardly follows from Theorem~\ref{reg} that this conjecture is true for closed graphs. 

In \cite{Sara2} it is conjectured that $\reg(S/J_G)\leq r$ where $r$ is the number of maximal cliques of $G.$ It is obvious that this conjecture follows from 
\cite[Theorem 1.1]{MM} in case that $r\geq n-1.$ Therefore, one should look at the case when $r\leq n-1.$ This is the case, for instance, if $G$ is chordal.

We give a positive answer to Madani-Kiani  conjecture \cite{Sara2} for a class of chordal graphs  which includes trees; see Theorem~\ref{regchordal}. In particular, we 
derive that Matsuda-Murai conjecture \cite{MM} holds for trees; see Corollary~\ref{corolarMM}. Moreover, this implies, in particular that, for chordal graphs, 
Matsuda-Murai conjecture follows from Madani-Kiani conjecture. Indeed, let us assume that the latter conjecture is true and  that $G$ is a chordal graph 
such that $\reg(S/J_G)=n-1$. This implies that $G$ has $n-1$ cliques, that is, $G$ is a tree with maximal regularity. By Corollary~\ref{corolarMM} it follows that 
$G$ must be a line graph.

\section{Preliminaries}

In this section we  review notation and  fundamental results on binomial edge ideals that will be used in the next sections.
\subsection{Closed graph and its clique complex. }
Let $G$ be a simple graph on the vertex set $[n]$ and $S=K[x_1,\ldots,x_n,y_1,\ldots,y_n]$ the polynomial ring in $2n$ indeterminates over a field $K$ endowed with the lexicographic order  induced by $x_1>\cdots >x_n>y_1>\cdots >y_n.$ Let $J_G=(f_{ij}: \{i,j\}\in E(G))\subset S$ be the associated binomial edge 
ideal of $G$.  It is obvious that, for computing the regularity of $S/J_G,$ we may assume without loss of generality that $G$ has no isolated vertex. According 
to \cite{HHHKR}, $G$ is called {\em closed  with respect to its given labeling} if its generators form a Gr\"obner basis with respect to the lexicographic order 
induced by the natural ordering of the indeterminates. $G$ is {\em closed} if it possesses a labeling of the vertices with respect to which is closed. In \cite{HHHKR} it is shown that any closed graph is chordal and claw--free.  In
\cite[Theorem 2.2]{EHH} it is shown that $G$ is closed if and only if there exists a labeling of $G$ such that all the facets of the clique complex $\Delta(G)$ 
of $G$ are intervals $[a,b]\subset [n].$ This means, in particular, that if $\{i,j\}\in E(G)$, then for any $i\leq  k< \ell \leq j$, $\{k,\ell\}\in E(G).$ 
Indeed, if $F$ is a facet of $\Delta(G)$ which contains $\{i,j\}$, then $F$ contains all the edges $\{k,\ell\}$ where $i\leq  k< \ell \leq j$. 

Let $G$ be  a closed graph. Then $\ini_{\lex}(J_G)=(x_iy_j: \{i,j\}\in E(G))$ is the edge ideal of a bipartite graph on the vertex set 
$\{x_1,\ldots,x_n\}\cup\{y_1,\ldots,y_n\}.$ Inspired by the notation in \cite{Sara}, we let $\ini_{\lex}(G)$ to be this bipartite graph. Therefore, we have 
$\ini_{\lex}(J_G)=I(\ini_{\lex}(G)).$

\subsection{Minimal primes of a binomial edge ideal.} Let  $\Sc\subset [n]$ and $G_1,\ldots,G_{c(\Sc)}$ be the connected components of the induced graph of $G$ 
on  the vertex set $[n]\setminus \Sc.$ For each $G_i$, let $\tilde{G}_i$ to be the complete graph on the vertex set of $G_i.$ Then, by \cite[Section 3]{HHHKR}, 
$$P_{\Sc}(G)=(\{x_i,y_i\}_{i\in \Sc}, J_{\tilde{G}_1},\ldots,J_{\tilde{G}_{c(\Sc)}})$$ is a prime ideal   for any subset 
$\Sc\subset [n]$, and $$J_G=\bigcap_{\Sc\subset [n]}P_{\Sc}(G).$$ Moreover, if $G$ is connected, then the minimal primes of $J_G$ are those ideals $P_{\Sc}(G)$ 
which correspond to the sets with the cut-point property \cite[Corollary 3.9]{HHHKR}. We say that $\Sc$ has the cut-point property or, simply, $\Sc$ is a {\em cut-point set} if 
$\Sc=\emptyset$ or $\Sc\neq \emptyset$ and $c(\Sc\setminus\{i\})< c(\Sc)$ for every $i\in \Sc.$

\section{Regularity}
\label{regsection}

We recall that, for an arbitrary simple graph $H,$ $\indmatch(H)$ is the number of edges in a largest induced matching of $H.$ By an {\em induced matching} we mean an induced 
subgraph of $H$ which consists of pairwise disjoint edges. Note that $\indmatch(H)$ is actually the monomial grade of the edge ideal $I(H)$ which is 
 the maximum length of
a regular sequence of monomials in $I(H)$. 

A graph $H$ is called {\em weakly chordal} if every induced cycle in $H$ and $\overline{H}$ (the complementary graph of $H$) has length at most $4.$

We recall the following theorem in \cite{W}.

\begin{Theorem}[\cite{W}]\label{Wood}
If $H$ is a weakly chordal graph on the vertex set $[n],$ then $$\reg(K[x_1,\ldots,x_n]/I(H))=\indmatch(H).$$
\end{Theorem}

\begin{Theorem}\label{reg}
Let $G$ be a closed graph on the vertex set $[n]$ with the connected components $G_1,\ldots,G_r.$ Then 
\[
\reg(S/G_G)=\reg(S/\ini_<(J_G))) =\ell_1+\cdots+ \ell_r,
\]
where, for $1\leq i\leq r,$ $\ell_i$ is the length of the longest induced path of $G_i.$
\end{Theorem}

For the proof of this theorem we first need the following results.

\begin{Lemma}\label{weakly}
Let $G$ be a connected closed graph on $[n]$. Then the bipartite graph $H=\ini_{\lex}(G)$ on the vertex set $\{x_1,\ldots,x_n\}\cup\{y_1,\ldots,y_n\}$ is weakly 
chordal.
\end{Lemma}

\begin{proof}
It is almost obvious that $\overline{H}$ has no induced cycle of length $\geq 5.$ Indeed, this is due to the fact that $\overline{H}$ consists of two complete 
graphs, say $K_n^x$ on the vertex set $\{x_1,\ldots,x_n\}$ and $K_n^y$ on the vertex set $\{y_1,\ldots,y_n\}$, together with the edges $\{x_iy_j: i\geq j\}
\cup\{x_iy_j: i<j,\ \{i,j\}\notin E(G)\}$. Hence, if $C$ is an induced cycle of $\overline{H}$ of length $\geq 5,$ then $C$ contains at least three vertices either from 
$K_n^x$, or from $K_n^y$, thus it cannot be an induced cycle of $\overline{H}.$

Let us now prove that $H$ has no induced cycle of length $\geq 5.$ Assume that this is not the case, and choose $C$ with vertices $x_{i_1},y_{j_1},\ldots,
x_{i_k},y_{j_k}$, $k\geq 3,$ an induced cycle of $H.$ This means that $\{i_\ell, j_\ell\}$ and $\{i_{\ell+1}, j_\ell\}$ are edges of $G$ for $1\leq \ell\leq k,$ 
where we made the convention that $i_{k+1}=i_1.$ We may assume that $i_1<i_2.$ If there exists $\ell$ such that $i_\ell< j_{\ell+1}<j_\ell,$ as $G$ is closed, it 
follows that $\{i_\ell,j_{\ell+1}\}\in E(G)$ which implies that $\{x_{i_\ell},y_{j_{\ell+1}}\}\in E(H)$, a contradiction, since $C$ is an induced subgraph of $H.$ 
Therefore, for all $\ell,$ we must have either $i_\ell <i_{\ell+1}<j_\ell <j_{\ell+1}$ or $i_{\ell+1}< j_{\ell+1}\leq i_\ell <j_\ell.$ As $i_1< i_2,$ we may 
choose $t$ to be the largest index such that $i_t<i_{t+1}.$ Thus, we get $i_t< i_{t+1}<j_t<j_{t+1}$ and $i_{t+2}<j_{t+2}\leq i_{t+1}<j_{t+1},$ which implies that 
$i_{t+2}<j_{t+2}\leq i_{t+1}<j_t<j_{t+1}.$ Since $\{i_{t+2},j_{t+1}\}\in E(G)$ and $G$ is closed, we obtain $\{i_{t+2},j_t\}\in E(G)$ which leads to 
$\{x_{i_{t+2}},y_{j_t}\}\in E(H)$, again a contradiction to the choice of $C$.
\end{proof}

\begin{Corollary}\label{regin}
Let $G$ be a  closed graph on $[n]$ and $H=\ini_{\lex}(G)$. Then $$\reg(S/I(H))=\indmatch(H).$$
\end{Corollary}

\begin{Proposition}\label{mainprop}
Let $G$ be a connected closed graph on $[n]$ and $H=\ini_{\lex}(G)$. Then $\indmatch(H)=\ell,$ where $\ell$ is the length of the longest induced path of $G.$
\end{Proposition}

\begin{proof}
First we show that $\indmatch(H)\geq \ell.$ This follows easily since it is obvious that if $i_0,\ldots,i_{\ell}$ is an induced path in $G$ of length $\ell,$
then the edges $\{x_{i_0},y_{i_1}\}, \{x_{i_1},y_{i_2}\},\ldots,\{x_{i_{\ell-1}},y_{i_\ell}\}$ form an induced subgraph of $H.$

We show now that $\indmatch(H)\leq \ell.$ Let $\indmatch(H)=m.$ Then $H$ has $m$ pairwise disjoint edges $\{x_{i_1},y_{j_1}\},\ldots, \{x_{i_m},y_{j_m}\}$ that 
form an induced subgraph of $H.$ We may assume that $i_1<\cdots < i_m.$ To show the desired inequality we  construct a path of length $m$ in $G$.

As $G$ is closed, we may assume, as we have seen in Preliminaries, that all the facets of the clique complex of $G$ are intervals. This implies that, for 
$s=1,\ldots, m,$ we may choose the smallest index $i_s^\prime$ among the indices $t$ such that $\{x_t,y_{j_s}\}\in E(H)$ and \[\{x_{i_1^\prime},y_{j_1}\},\ldots, 
\{x_{t},y_{j_s}\},\{x_{i_{s+1}},y_{j_{s+1}}\},\ldots,\{x_{i_m},y_{j_m}\}\] is an induced matching of $H.$  

By these arguments, we may assume already from the beginning that the induced subgraph of $H$ with $m$ pairwise distinct edges, $\{x_{i_1},y_{j_1}\},\ldots, 
\{x_{i_m},y_{j_m}\}$,  satisfies the following condition: 
\[(\ast)\ \ \ \ \text{ for all } 1\leq s\leq m, \text{ if } t<s \text{ and }  \{t,j_s\}\in E(G), 
\]
\[
\text{ then } 
\{x_{i_1},y_{j_1}\},\ldots,\{ x_t,y_{j_s}\},\{x_{i_{s+1}},y_{j_{s+1}}\},\ldots,\{x_{i_m},y_{j_m}\} \text{ is not an induced matching of } H.
\]
Note that we also have $j_t\leq i_{t+1}$ for all $1\leq t\leq m-1.$ Indeed, if there exists $t$ such that $j_t>i_{t+1},$ then it follows that $i_t<i_{t+1}<j_t$. 
We obtain  $\{i_{t+1},j_t\}\in E(G)$ and $\{x_{i_{t+1}},y_{j_t}\}\in E(H)$, a contradiction to our hypothesis. 

In the second part of the proof we show that, under  condition ($\ast$) for the induced subgraph $\{x_{i_1},y_{j_1}\},$ $\ldots, 
\{x_{i_m},y_{j_m}\}$ of $H$, we have:
\begin{enumerate}
	\item [(i)] $i_s$ and $i_{s+1}$ belong to the same clique of $G$ for all $1\leq s\leq m-1,$  
	\item [(ii)] $i_s,i_{s+1},i_{s+2}$ do not belong to the same clique for any $1\leq s\leq m-2.$
\end{enumerate}

Let us assume that we have already shown (i) and (ii). Then $L: i_1,i_2,\ldots,i_m,j_m$ is an induced path of $G.$ Indeed, by (i), $L$ is a path in $G$. Next, it 
is clear that we cannot have an 
edge $\{i_s,i_q\}\in E(G)$ with $q-s\geq 2$ by (ii). In addition, by ($\ast$), it follows that $\{i_s,j_m\}\notin E(G)$ for any 
$1\leq s\leq m-1.$ Therefore, $L$ is an induced path of $G.$

Let us first  prove (ii). Suppose that there are three consecutive vertices $i_s,i_{s+1}, i_{s+2}$ in $L$ which belong to the same clique of $G.$ Hence $\{i_s,i_{s+2}\}\in E(G).$ As
$i_s< j_s\leq i_{s+1}< j_{s+1}\leq i_{s+2}<j_{s+2}$, we also have $\{i_s,j_{s+1}\}\in E(G),$ which is impossible. 

Finally, we show (i). Let us assume that there exists $s$ such that $i_s$ and $i_{s+1}$ do not belong to the same clique of $G$, in other words, 
$\{i_s,i_{s+1}\}\notin E(G).$ In particular, we have $j_s <i_{s+1}$, thus $i_s<j_s <i_{s+1}$. We need to consider the following two cases.

{\em Case} (a). $\{j_s,i_{s+1}\}\in E(G).$ We claim that 
\[
\{x_{i_1},y_{j_1}\},\ldots,\{x_{i_s},y_{j_s}\}, \{x_{j_s},y_{i_{s+1}}\}, \{x_{i_{s+1}},y_{j_{s+1}}\}, \ldots, \{x_{i_m},y_{j_m}\}
\]
is an induced subgraph of $H$ with pairwise distinct edges. This will lead to contradiction since $\indmatch(H)=m.$ To prove our claim, we note that 
$\{x_{j_s},y_{j_{s+1}}\}\notin E(H)$ by ($\ast$) and $\{x_{i_s},y_{i_{s+1}}\}\notin E(H)$ since $\{i_s,i_{s+1}\}\notin 
E(G).$ Moreover, if $\{x_{i_q},y_{i_{s+1}}\}\in E(H)$ for some $q<s,$ then, as we have $i_q<i_s<j_s<i_{s+1},$ we get $\{i_q,j_s\}\in E(G)$, thus 
$\{x_{i_q},y_{j_s}\}\in E(H)$, contradiction. Similarly, if $\{x_{j_s},y_{j_q}\}\in E(H)$ for some $q\geq s+2,$ as $j_s< i_{s+1}<i_q<j_q$, we get
$\{i_{s+1},j_q\}\in E(G)$, that is, $\{x_{i_{s+1}},y_{j_q}\}\in E(H)$, contradiction.

{\em Case} (b). $\{j_s,i_{s+1}\}\notin E(G).$ Let then $j=\min\{t: \{t,{i_{s+1}}\}\in E(G)\}$. Since $G$ is closed, we must have $j> j_s> i_s$. Let us consider the following 
disjoint edges of $H:$ 
\[
\{x_{i_1},y_{j_1}\},\ldots,\{x_{i_s},y_{j_s}\}, \{x_{j},y_{i_{s+1}}\}, \{x_{i_{s+1}},y_{j_{s+1}}\}, \ldots, \{x_{i_m},y_{j_m}\}.
\]
These edges determine an induced subgraph of $H,$which leads again to a contradiction to the fact that $\indmatch(H)=m.$ Indeed, since $j< i_{s+1}$, it follows that 
$\{j,j_{s+1}\}\notin E(G)$. As in the previous case, we get $\{x_{i_q},y_{i_{s+1}}\}\notin E(H)$ for $q<s$. Let us assume that $\{x_j,y_{j_q}\}\in E(H)$
for some $q\geq s+2.$ Then $\{j,j_q\}\in E(G)$ and since $j< i_{s+1}<j_{s+1}<j_q,$ we get $\{i_{s+1},j_q\}\in E(G)$ or, equivalently, $\{x_{i_{s+1}},y_{j_q}\}\in 
E(H)$, impossible.
\end{proof}

\begin{proof}[Proof of Theorem \ref{reg}] It is obvious that if $S_i$ is the polynomial ring in the indeterminates indexed by the vertex set of $G_i$ for $1\leq 
i\leq r,$ then $S/J_G\cong \bigotimes_{i=1}^r (S_i/J_{G_i})$. Therefore, $\reg(S/J_G)=\sum_{i=1}^r \reg(S_i/J_{G_i})$. This equality shows that it is enough to prove 
our statement for a connected closed graph $G.$ Therefore, for the remaining part of the proof we take $G$ to be connected and show that $\reg(S/J_G)=\ell,$ 
where 
$\ell$ is the length of the longest induced path of $G.$ 

The inequality $\reg(S/J_G)\geq \ell$ is known by \cite[Theorem 1.1]{MM}. On the other hand, by using \cite[Theorem 3.3.4]{HH10} and Corollary~\ref{regin}, we get 
\[
\reg(S/J_G)\leq \reg(S/\ini_{\lex}(J_G))=\reg(S/I(H))= \indmatch(H)
\]
where $H=\ini_{lex}(G).$ But, by Proposition~\ref{mainprop}, $\indmatch(H)=\ell$, and this ends our proof.
\end{proof}

\begin{Corollary}\label{char}
Let $G$ be a closed graph. Then the regularity of $J_G$ and of $\ini_{\lex}(J_G)$ do not depend on the characteristic of the base field.  
\end{Corollary}

In \cite{MM}, Matsuda and Murai conjectured that if $G$ is a connected graph on $n$ vertices, then $\reg(S/J_G)=n-1$ if and only if $G$ is a line graph.  Theorem~\ref{reg}
gives a positive answer to this conjecture for closed graphs.

\begin{Corollary}
Let $G$ be a connected closed graph on $[n]$. Then $\reg(S/J_G)=n-1$ if and only if $G$ is a line graph.
\end{Corollary}


In \cite{Sara2}, the following conjecture is proposed.

\begin{Conjecture}\label{Saraconj}
Let $G$ be a graph. Then $\reg(S/J_G)\leq r,$ where $r$ is the number of maximal cliques of $G.$
\end{Conjecture}

Note that, for chordal graphs, this conjecture  implies Matsuda and Murai's conjecture, once we show that the latter is true for trees. Indeed, if $G$ is a connected chordal graph with 
$\reg(S/J_G)=n-1,$ then, if Conjecture~\ref{Saraconj} is true,  it follows that $n-1\leq r,$ thus $G$ must be a tree.  

In the sequel, 
 we first prove Conjecture~\ref{Saraconj} for the binomial edge ideals associated with a special class of chordal graphs  
introduced in \cite[Section 1]{EHH}, and, afterwards, we show that  Matsuda and Murai's 
conjecture  is true for this special class of graphs which includes trees. 

\begin{Theorem}\label{regchordal} 
Let $G$ be a  chordal graph on $[n]$ with the property that any two distinct maximal cliques intersect in  at most one  vertex. Then 
$\reg(S/J_G)\leq r$ where $r$ is the number of maximal cliques of $G.$
\end{Theorem}

\begin{proof} Obviously, by using the argument given in the beginning of the proof of Theorem~\ref{reg}, we may assume that $G$ is connected. We make induction 
on $r$ and  closely follow some arguments given in the proof of \cite[Theorem 1.1]{EHH}. For $r=1,$ the statement is well known. Let $r>1.$ As in the proof of 
\cite[Theorem 1.1]{EHH}, let us consider $F_1,\ldots, F_r$ a leaf order on the facets of the clique complex $\Delta(G)$ of $G.$ Since $F_r$ is a leaf, there 
exists a unique vertex, say $i\in F_r$, such that $F_r\sect F_j=\{i\}$  for some $j$. Let $F_{t_1},\ldots,F_{t_q}$ be the
facets of $\Delta(G)$ which intersect the leaf $F_r$ in the vertex $i.$  We decompose $J_G$ as $J_G=Q_1\sect Q_2$ where $Q_1$ is the intersection of all minimal 
primes $P_\Sc(G)$ of $J_G$ with $i\not\in \Sc$, and $Q_2$ is the intersection of all minimal primes $P_\Sc(G)$ with $i\in \Sc.$ This decomposition yields the 
following exact sequence of $S$-modules:

\begin{equation}\label{eqchordal}
0 \to S/J_G\to S/Q_1\oplus S/Q_2\to S/(Q_1+Q_2)\to 0.
\end{equation}

In the proof of \cite[Theorem 1.1]{EHH}, it was shown that $Q_1=J_{G^\prime}$ where $G^\prime$ inherits the properties of $G$ and has less than $r$ cliques. 
Hence, by the inductive hypothesis, $\reg(S/J_{G^\prime})<r. $ In the same proof it was shown that $Q_2=(x_i,y_i)+J_{G^{\prime\prime}}$ where $G^{\prime\prime}$ is a chordal 
graph on $n-1$ vertices with $q+1\geq 2$ components satisfying the conditions of the theorem and with at most $r$ cliques. Therefore,  we get $\reg(S/Q_2)\leq r$ and, consequently,  $\reg(S/Q_1\oplus S/Q_2)\leq r.$ Finally, as it was observed in  \cite[Theorem 1.1]{EHH}, $Q_1+Q_2=(x_i,y_i)+J_{G^\prime}$, hence 
$S/Q_1+Q_2\cong S_i/J_H$ for a graph $H$ on the vertex set $[n]\setminus\{i\}$ which has the same properties as $G$ and a number of cliques $<r.$ (Here $S_i$ is the polynomial ring in the variables indexed by $[n]\setminus\{i\}$.) Consequently, by induction, we have $\reg(S/Q_1+Q_2)<r.$ 

By applying \cite[Corollay 18.7]{P} to the sequence~(\ref{eqchordal}), we get $\reg(S/J_G)\leq r.$
\end{proof}

By using the notation and the arguments in the proof of the above theorem we also get the following statement.

\begin{Corollary}\label{corolarMM}
Let $G$ be be a  connected chordal graph on $[n]$ with the property that any two distinct maximal cliques intersect in  at most one  vertex. If $\reg(S/J_G)=n-1$, then $G$ is a line graph.
\end{Corollary}

\begin{proof}
By the above theorem, we have $\reg(S/J_G)\leq r.$ Hence $G$ has $n-1$ cliques, which implies that $G$ is a tree. By using the sequence~(\ref{eqchordal}) and 
\cite[Corollay 18.7]{P}, it follows that at least one of the following inequalities must hold: 
\begin{itemize}
	\item [(i)] $\reg(S/Q_1)\geq n-1,$
	\item [(ii)] $\reg(S/Q_2)\geq n-1,$
	\item [(iii)] $\reg(S/Q_1+Q_2)\geq n-2.$
\end{itemize}
	The first inequality cannot hold since the graph $G^\prime$ has a smaller number of cliques than $G.$ The second one is also impossible since it would imply that $\reg(S/Q_2)=\reg(S_i/J_{G^{\prime\prime}})\geq n-1$, and $G^{\prime\prime}$ is a graph on $n-1$ vertices. Therefore, we must have (iii). Since 
	$\reg(S/Q_1+Q_2)=\reg(S_i/J_{H})$ where $H$ is obtained form $G'$ by replacing the clique on the vertex set $F_r\union(\Union\limits_{j=1}^q F_{t_j})$ by the clique on the vertex set $F_r\union(\Union\limits_{j=1}^q F_{t_j})\setminus \{i\}$ (see the proof of \cite[Theorem 1.1]{EHH}), by induction on $n$, it follows that $H$ must be a line graph. This implies that the clique on $F_r\union(\Union\limits_{j=1}^q F_{t_j})\setminus \{i\}$ has $2$ vertices. In particular, it follows that  $q=1$ and $G$ is a line graph as well.
\end{proof}

{}

\end{document}